\documentclass{article}

\usepackage{arxiv}

\usepackage[utf8]{inputenc} 
\usepackage[T1]{fontenc}    
\usepackage{hyperref}       
\usepackage{url}            
\usepackage{booktabs}       
\usepackage{amsfonts}       
\usepackage{nicefrac}       
\usepackage{microtype}      
\usepackage{lipsum}         
\usepackage{graphicx}
\usepackage{doi}

\usepackage{amssymb,amsthm,amsmath}
\usepackage{xcolor,paralist,hyperref,titlesec,fancyhdr,etoolbox}
\newtheorem{theorem}{Theorem}[]
\newtheorem{definition}{Definition}

\newtheorem{lemma}{Lemma}

\newtheorem*{remark}{Remark}

\title{Explicit construction of infinite families of strongly regular digraphs with parameters $((v+(2^{n+1}-4)t)2^{n-1}, k+(2^n-2)t, t, \lambda, t)$\thanks{Translated from the Russian original published in: V. A. Byzov, I. A. Pushkarev, “Explicit construction of infinite families of strongly regular digraphs with parameters $((v+(2^{n+1}-4)t)2^{n-1}, k+(2^n-2)t, t, \lambda, t)$”, Prikladnaya Diskretnaya Matematika, 2025, no. 69, 111--120.}}

\date{September, 2025}

\author{ \href{https://orcid.org/0000-0002-3613-5949}{\includegraphics[scale=0.06]{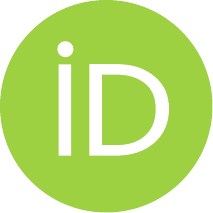}\hspace{1mm}Viktor A.~Byzov} \\
	Vyatka State University\\
	Kirov,  Russian Federation \\
	\texttt{vbyzov@yandex.ru} \\
	\And
	{ \href{https://orcid.org/0000-0002-3610-2872}{\includegraphics[scale=0.06]{orcid.pdf}\hspace{1mm}Igor A.~Pushkarev}} \\
	Vyatka State University\\
	Kirov,  Russian Federation \\
	\texttt{god\_sha@mail.ru}\\
}

\renewcommand{\headeright}{~}
\renewcommand{\undertitle}{~}
\renewcommand{\shorttitle}{Explicit construction of infinite families of strongly regular digraphs with parameters $((v+(2^{n+1}-4)t)2^{n-1}, k+(2^n-2)t, t, \lambda, t)$}

\begin{document}
\maketitle

\begin{abstract}
An explicit construction of infinite sequences of strongly regular digraphs with parameter sets $((v+(2^{n+1}-4)t)2^{n-1}, k+(2^n-2)t, t, \lambda, t)$ is described. A computer program was used to find the initial digraphs. The remaining terms of the sequence are obtained by the constructed recurrence. Using the described approach, 11 families of strongly regular digraphs were found. In particular, these families contain digraphs~$\text{dsrg}(72, 18, 5, 3, 5)$, $\text{dsrg}(76, 19, 5, 4, 5)$, $\text{dsrg}(92, 23, 6, 5, 6)$ and~$\text{dsrg}(104, 26, 7, 5, 7)$, the question of the existence of which was previously open.
\end{abstract}


\vspace{10pt}
\section{Introduction}

This paper describes the method by which the authors constructed several infinite families of digraphs. Before describing the method, we introduce the necessary concepts and notation.

We will be dealing with directed graphs (digraphs) without loops and without multiple edges in the same direction. The adjacency matrix of a digraph is, in the standard sense, the matrix in which a 1 is placed at the intersection of the $i$-th row and the $j$-th column if the digraph has an edge going from the vertex numbered $i$ to the vertex numbered $j$; otherwise, a 0 is placed in that entry of the matrix. Rows and columns are numbered with natural numbers.

We will also use the following standard notations for matrices of a special type:

\begin{enumerate}
	\item $I_m$ is identity matrix of order $m$;
	\item $J_m$ is a square matrix of order $m$ consisting only of ones;
	\item $J_{m, l}$ is a rectangular matrix of size $m\times l$, consisting only of ones;
	\item $K_m$ is an exchange matrix of order $m$, that is, a square matrix of order $m$ in which the antidiagonal contains ones and all other positions contain zeros. In other words, $K_m(i, j) = \delta_{m+1-i, j}$, where $\delta$ is the Kronecker delta.
\end{enumerate} 

By $A \otimes B$ we will denote the Kronecker product of matrices $A$ and $B$.

The concept of a strongly regular digraph is a generalization of the concept of a strongly regular graph (see, for example,~\cite{bibBose1, bibSRG1}). To the best of our knowledge, it was first introduced by A.\,M.~Duval in the paper~\cite{bibDuval1}.

\begin{definition}
	\label{def:main1}
	A strongly regular directed graph with parameter set ($v$, $k$, $t$, $\lambda$, $\mu$) is a digraph on $v$ vertices in which the following conditions hold:
	\begin{enumerate}
	\item the outdegree and indegree of each vertex are equal to $k$;
	\item for each vertex $x$ there exist exactly $t$ paths of length two from $x$ to $x$;
	\item for each edge $x \to y$, the number of directed paths of length two from $x$ to $y$ is equal to $\lambda$;
	\item for each ordered pair of vertices $(x, y)$ that does not form an edge of the digraph, the number of directed paths of length two from $x$ to $y$ is equal to $\mu$.
	\end{enumerate}
\end{definition}

An equivalent definition of a strongly regular digraph can be given through a set of conditions imposed on its adjacency matrix.

\begin{definition}
	\label{def:main2}
	A strongly regular directed graph with parameter set ($v$, $k$, $t$, $\lambda$, $\mu$) is a digraph whose adjacency matrix $A$ satisfies the following conditions:
    \begin{equation}
	\label{eq:main1}
	A^2 = t I_v + \lambda A + \mu (J_v - I_v - A),
	\end{equation}
	\begin{equation}
	\label{eq:main2}
	A J_v = J_v A = k J_v.
	\end{equation}
\end{definition}

For brevity, instead of the phrase ``strongly regular digraph with parameter set ($v$, $k$, $t$, $\lambda$, $\mu$)'', we shall use the notation $\text{dsrg}(v, k, t, \lambda, \mu)$, as is often done. It should be noted that when employing this notation in the present paper, we refer not to the entire set of such digraphs, but to a particular one --- namely, the digraph realized within the framework of the construction being described.

A number of necessary conditions for the existence of strongly regular digraphs with parameter set ($v$, $k$, $t$, $\lambda$, $\mu$) are presented in several works (see, for example, \cite{bibDuval1, bibJorgensen1}, among others). However, as in the case of strongly regular graphs, for many parameter sets, the question of the existence of a corresponding strongly regular digraph remains open. The table on the website~\cite{bibBrouwer1} systematizes information about known strongly regular digraphs and those parameter sets ($v$, $k$, $t$, $\lambda$, $\mu$) for which the question of their existence has not yet been resolved.

This paper describes a method that (with a certain amount of luck) can be used to construct an infinite family of strongly regular digraphs satisfying the condition $\mu = t$. The digraphs are built recursively: each subsequent digraph is obtained from the preceding member of the sequence. The parameters $t$, $\lambda$, and $\mu$ are the same for all digraphs in the sequence, while the number of vertices and the degrees of the vertices increase. We denote by $G_1, G_2, \ldots$ the strongly regular digraphs in the constructed sequence, and by $A_1, A_2, \ldots$ the corresponding adjacency matrices. A significant difficulty of the described method lies in obtaining the digraph $G_2$ from $G_1$: it is necessary for there to exist matrices satisfying a certain set of conditions. There are no guarantees, however, that such matrices exist at all. At the current stage, the authors are searching for these matrices using a computer program.

The paper is organized as follows. The first section describes the method of obtaining the matrix $A_2$ from the matrix $A_1$. The second section proves that, under the condition of the existence of matrices $A_1$ and $A_2$ (of the required form), an infinite sequence of strongly regular digraphs can be constructed. The third section presents families of digraphs found with the help of a computer program.

\section{Obtaining the digraph $G_2$}
\label{sec1}

For further discussion we will need an auxiliary sequence of matrices $P_n$, defined as follows:

\begin{equation}
\label{eq:Pndef}
P_n = J_{2^n, 1} \otimes K_{2^n} \otimes J_{t, t\cdot 2^n}.
\end{equation}

This sequence of matrices has a number of useful properties. Specifically, the following lemma holds.

\begin{lemma}
\label{lemma:propPn}
The matrix $P_n$, defined by formula~\eqref{eq:Pndef}, is a square matrix of order $t \cdot 4^n$ possessing the following properties:
\begin{enumerate}
	\item $P_n J_{t\cdot 4^n} = J_{t\cdot 4^n} P_n = t\cdot2^n J_{t\cdot 4^n}$;
	\item $P_n^2 = t J_{t\cdot 4^n}$.
\end{enumerate} 
\end{lemma}

\begin{proof}
The validity of the first property is almost obvious: in each row and in each column of the matrix $P_n$ there are, by construction, exactly $t\cdot 2^n$ ones.

To prove the second point of the lemma, we will use properties of the Kronecker product (see, for example,~\cite{bibZhang1}):
\begin{equation}
\begin{split}
P_n^2 &= \left((J_{2^n, 1} \otimes K_{2^n}) \otimes J_{t, t\cdot 2^n}) \cdot (J_{2^n, 1} \otimes (K_{2^n} \otimes J_{t, t\cdot 2^n})\right) = \left((J_{2^n, 1} \otimes K_{2^n}) \cdot J_{2^n, 1}\right) \otimes\\
&\otimes \left(J_{t, t\cdot 2^n} \cdot (K_{2^n} \otimes J_{t, t\cdot 2^n})\right) = J_{4^n, 1} \otimes tJ_{t, t\cdot 4^n} = tJ_{t\cdot 4^n}.
\end{split}
\end{equation}
\end{proof}

Let $G_1$ be a strongly regular digraph with parameter set $(v, k, t, \lambda, t)$, and let $A_1$ be its adjacency matrix. It is known (see~\cite{bibDuval1}) that the condition $t > \lambda$ must hold. Introduce the notation: $s = t - \lambda$ ($s > 0$). From Definition~\ref{def:main2} it follows that the matrix $A_1$ satisfies the following relations:
\begin{equation}
\label{eq:condsA1}
A_1^2 + sA_1 = tJ_v, \quad A_1 J_v = J_v A_1 = k J_v.
\end{equation}

We will look for the second member of the sequence of digraphs --- $\text{dsrg}(2v+8t, k+2t, t, \lambda, t)$. Denote the adjacency matrix of this digraph by $A_2$, and we will seek it in the following form:

\begin{equation}
\label{eq:block_reprA2}
A_2 = \left(
\begin{array}{cccc}
A_1 & 0 & B_1 & 0 \\
0 & A_1 & 0 & B_1 \\
0 & C_1 & 0 & P_1 \\
C_1 & 0 & P_1 & 0
\end{array}
\right),
\end{equation}
\noindent where $B_1$ and $C_1$ are unknown matrices of sizes $v \times 4t$ and $4t \times v$, respectively.

By Definition~\ref{def:main2}, the following relations must hold for the matrix $A_2$:
\begin{equation}
\label{eq:condsA2}
A_2^2 + sA_2 = tJ_{2v+8t}, \quad A_2 J_{2v+8t} = J_{2v+8t} A_2 = (k+2t) J_{2v+8t}.
\end{equation}

Using the block representation~\eqref{eq:block_reprA2} of the matrix $A_2$, rewrite the first equation~\eqref{eq:condsA2} in the following form:
\begin{equation}
\label{eq:block_repr_eq}
\left(
\begin{array}{c|c|c|c}
A_1^2+sA_1 & B_1 C_1 & A_1 B_1 + sB_1 & B_1 P_1 \\ \hline
B_1 C_1 & A_1^2+sA_1 & B_1 P_1 & A_1 B_1 + sB_1 \\ \hline
P_1 C_1 & C_1 A_1 + sC_1 & P_1^2 & C_1 B_1 + sP_1 \\ \hline
C_1 A_1 + sC_1 & P_1 C_1 & C_1 B_1 + sP_1 & P_1^2
\end{array}
\right) = tJ_{2v+8t}.
\end{equation}

Note that, from Lemma~\ref{lemma:propPn} and the first relation~\eqref{eq:condsA1}, it follows that the diagonal blocks of the resulting block matrix have the required form. Thus, the matrices $B_1$ and $C_1$ must satisfy the following system of matrix equations:
\begin{equation}
\label{eq:systemA2part1}
\begin{cases}
B_1 C_1 = tJ_v,\\
B_1 P_1 = tJ_{v, 4t},\\
P_1 C_1 = tJ_{4t, v},\\
A_1 B_1 + sB_1 = tJ_{v, 4t},\\
C_1 A_1 + sC_1 = tJ_{4t, v},\\
C_1 B_1 + sP_1 = tJ_{4t}.
\end{cases}
\end{equation}

Moreover, from the second equation~\eqref{eq:condsA2} and the form of the matrices $A_1$ and $P_1$, it follows that the matrices $B_1$ and $C_1$ must also satisfy the conditions
\begin{equation}
\label{eq:systemA2part2}
B_1 J_{4t} = 2tJ_{v, 4t}, \quad J_v B_1 = kJ_{v, 4t}, \quad C_1 J_{v} = kJ_{4t, v}, \quad J_{4t} C_1 = 2tJ_{4t, v}.
\end{equation}

Additionally, we add two further conditions that the matrices $B_1$ and $C_1$ must satisfy. These conditions are not necessary for the existence of the digraph $G_2$, but they will be needed for constructing the subsequent digraphs in the sequence.

\begin{enumerate}
\item If any row of the matrix $B_1$ is split in half into two blocks of length $2t$, then there is exactly one block consisting entirely of ones and exactly one block consisting entirely of zeros.
\item If any column of the matrix $C_1$ is split in half into two blocks of length $2t$, then each block contains exactly $t$ ones.
\end{enumerate}

For convenience, call these constraints imposed on the matrices $B_1$ and $C_1$ \textit{the blockiness conditions} of these matrices.

\begin{remark}
If the matrices $B_1$ and $C_1$ satisfy the blockiness conditions, then the first three equalities in~\eqref{eq:systemA2part1} hold automatically. This follows from the fact that the matrix $P_1$ also satisfies both of the above conditions.
\end{remark}

Thus, the sought matrices $B_1$ and $C_1$ must satisfy the relations in~\eqref{eq:systemA2part1} and~\eqref{eq:systemA2part2} as well as the blockiness conditions. If such matrices are found, then we obtain the adjacency matrix $A_2$ of the strongly regular digraph $G_2$. In the next section, it will be shown how to continue this sequence of digraphs.

\section{Recurrent formulas for constructing a family of digraphs $\text{dsrg}((v+(2^{n+1}-4)t)2^{n-1}, k+(2^n-2)t, t, \lambda, t)$}
\label{sec2}

In this section, a new operation on matrices will be needed. Let $X$ be an $m\times l$ matrix. Denote by $\alpha_s(X)$ the matrix obtained by taking the first $\frac{m}{s}$ rows of $X$ (it is assumed that $m$ is divisible by $s$).

After introducing this operation, a recursive definition of the sequence of matrices $P_n$ can be given.
\begin{lemma}
\label{lemma:Pn_newdef}
Let $P'_n$ be the sequence of matrices defined as follows:
\begin{equation}
P'_1 = J_{2, 1} \otimes K_2 \otimes J_{t, 2t}, \quad P'_n = J_{2^n, 1} \otimes K_2 \otimes \alpha_{2^{n-1}}(P'_{n-1}) \otimes J_{1, 2}.
\end{equation}
Then the sequence $P'_n$ coincides with the previously defined sequence of matrices $P_n$.
\end{lemma}

\begin{proof}
Using induction on $n$. The base case holds: $P'_1 = P_1$. Assume $P'_n = P_n$. Using the relation~\eqref{eq:Pndef}, transform $P'_{n+1}$:
\begin{multline}
P'_{n+1} = J_{2^{n+1}, 1} \otimes K_2 \otimes \alpha_{2^n}(P'_n) \otimes J_{1, 2} = \\ =
J_{2^{n+1}, 1} \otimes K_2 \otimes \alpha_{2^n}(J_{2^n, 1} \otimes K_{2^n} \otimes J_{t, t\cdot 2^n}) \otimes J_{1, 2} = J_{2^{n+1}, 1} \otimes K_2 \otimes K_{2^n} \otimes J_{t, t\cdot 2^n} \otimes J_{1, 2} = \\ =
J_{2^{n+1}, 1} \otimes K_{2^{n+1}} \otimes J_{t, t\cdot 2^{n+1}} = P_{n+1}.
\end{multline}
Thus, the inductive step is proved.
\end{proof}

Assume that matrices $B_1$ and $C_1$ have been found that satisfy the relations in~\eqref{eq:systemA2part1} and~\eqref{eq:systemA2part2} and the blockiness conditions (the sizes of these matrices are $v\times 4t$ and $4t\times v$, respectively). Based on $B_1$ and $C_1$, construct two sequences of matrices $B_n$ and $C_n$ according to the following recursive rules (the block representations of matrices are used here):
\begin{equation}
\label{eq:recBn}
B_n = \left(
\begin{array}{c}
K_2 \otimes B_{n-1} \otimes J_{1, 2} \\ \hline
I_2 \otimes P_{n-1} \otimes J_{1, 2}
\end{array}
\right) \; \text{for } \; n\geqslant 2.
\end{equation}
\begin{equation}
\label{eq:recCn}
\begin{split}
C_n& = \left(
\begin{array}{c|c}
J_{2^n, 1} \otimes I_2 \otimes \alpha_{2^{n-1}}(C_{n-1}) & J_{2^n, 1} \otimes I_2 \otimes \alpha_{2^{n-1}}(P_{n-1})
\end{array}
\right) \; \text{for } \; n\geqslant 2.
\end{split}
\end{equation}

It is straightforward to show that the matrix $B_n$ has size $(v+(2^{n+1}-4)t)2^{n-1} \times t\cdot 4^n$, and the matrix $C_n$ has size $t\cdot 4^n \times (v+(2^{n+1}-4)t)2^{n-1}$. This can be proved by induction on $n$.

\begin{remark}
Note that the operation $\alpha_{2^{n-1}}$ in the definition of the sequence $C_n$ is valid, since the number of rows of the matrix $C_{n-1}$ is divisible by $2^{n-1}$ (the validity of the operation is proved by induction). Moreover, it will be seen from what follows that in the definition of the operation $\alpha_{2^n}$ as applied to the sequence $C_n$, one could take not the first block of rows of the matrix, but any other.
\end{remark}

\begin{lemma}
\label{lemma:summBnCn}
The matrices $B_n$ and $C_n$ have the following properties:
\begin{enumerate}
	\item in each row of the matrix $B_n$ there are exactly $t\cdot 2^n$ ones;
	\item in each column of the matrix $B_n$ there are exactly $k + (2^n - 2)t$ ones;
	\item in each row of the matrix $C_n$ there are exactly $k + (2^n - 2)t$ ones;
	\item in each column of the matrix $C_n$ there are exactly $t\cdot 2^n$ ones.			
\end{enumerate}
\end{lemma}

\begin{proof}
These statements can be proved by induction on $n$. The base case at $n=1$ holds (see~\eqref{eq:systemA2part2}). The inductive step is straightforward: the sequences $B_n$ and $C_n$ are defined recursively via the Kronecker product, which makes it possible to track how the number of ones in the rows and columns of the matrices changes. In the course of the proof, it is necessary to use the first point of Lemma~\ref{lemma:propPn} (it contains information about the number of ones in the rows and columns of the matrix $P_n$).
\end{proof}

For the constructed sequences of matrices, the following lemma also holds.
\begin{lemma}$~$
\label{lemma:propsBnCnPn}
\begin{enumerate}
\item The matrices $B_n$ and $P_n$ have the following property: if any row of these matrices is partitioned into consecutive blocks of length $t\cdot 2^n$, then there is exactly one block consisting entirely of ones, and all the remaining blocks consist entirely of zeros.
\item The matrices $C_n$ and $P_n$ have the following property: if any column of these matrices is partitioned into consecutive blocks of length $t\cdot 2^n$, then each block contains exactly $t$ ones.
\end{enumerate}
\end{lemma}

\begin{proof}
First, note that the matrix sizes change in such a way that the partition of rows (columns) into blocks of the required length can always be performed correctly.

The validity of the stated properties for the matrix $P_n$ follows directly from the formula~\eqref{eq:Pndef}.

The fact that the sequences $B_n$ and $C_n$ satisfy the stated properties can be proved by induction on $n$. The base case at $n=1$ holds, since by assumption the matrices $B_1$ and $C_1$ satisfy the blockiness conditions. The inductive step follows from the form of the recurrence relations~\eqref{eq:recBn} and~\eqref{eq:recCn}. Indeed, if a row of the matrix $B_{n-1}$ (the matrix $P_{n-1}$) contained a block of ones, then in the matrix $K_2 \otimes B_{n-1} \otimes J_{1, 2}$ (in the matrix $I_2 \otimes P_{n-1} \otimes J_{1, 2}$) it is transformed into a block with twice as many ones. Suppose that in the columns of the matrix $C_{n-1}$ (the matrix $P_{n-1}$) each block contains $t$ ones. Then the matrix $\alpha_{2^{n-1}}(C_{n-1})$ (the matrix $\alpha_{2^{n-1}}(P_{n-1})$) contains the first blocks of columns, and the columns of the matrix $J_{2^n, 1} \otimes I_2 \otimes \alpha_{2^{n-1}}(C_{n-1})$ (the matrix $J_{2^n, 1} \otimes I_2 \otimes \alpha_{2^{n-1}}(P_{n-1})$) can be partitioned into blocks of twice the length, which still contain $t$ ones.
\end{proof}

Let's move on to the main result of the work.

\begin{theorem}
\label{th:main}
Let $A_n$ be a sequence of matrices in which $A_1$ is the adjacency matrix of the digraph $\text{dsrg}(v, k, t, \lambda, t)$, and for $n\geqslant 1$ it holds that
\begin{equation}
\label{eq:block_reprAnp1}
A_{n+1} = \left(
\begin{array}{cccc}
A_n & 0 & B_n & 0 \\
0 & A_n & 0 & B_n \\
0 & C_n & 0 & P_n \\
C_n & 0 & P_n & 0
\end{array}
\right)
= \left(
\begin{array}{cc}
I_2\otimes A_n & I_2\otimes B_n \\
K_2\otimes C_n & K_2\otimes P_n
\end{array}
\right)
,
\end{equation}
\noindent where $B_n$, $C_n$, and $P_n$ are the previously defined sequences (it is assumed that the matrices $B_1$ and $C_1$ exist).

Then the matrices $A_n$ are adjacency matrices of directed strongly regular digraphs with parameter sets~\mbox{$((v+(2^{n+1}-4)t)2^{n-1}, k+(2^n-2)t, t, \lambda, t)$}.
\end{theorem}

\begin{proof}
Introduce a notation for the order of the matrix $A_n$ (it is straightforward to show that it will indeed be this):
\begin{equation}
    v_n = (v+(2^{n+1}-4)t)2^{n-1}.
\end{equation}

It is necessary to prove two statements (see Definition~\ref{def:main2}):
\begin{enumerate}
	\item in each row and in each column of the matrix $A_n$ there are exactly $k+(2^n-2)t$ ones;
	\item $A_n^2 + sA_n = tJ_{v_n}$.
\end{enumerate}

Prove the first statement by induction on $n$. For the matrices $A_1$ and $A_2$ this fact holds. Assume that in each row and each column of $A_n$ there are exactly $k+(2^n-2)t$ ones. Then, using Lemma~\ref{lemma:propPn}, Lemma~\ref{lemma:summBnCn}, and formula~\eqref{eq:block_reprAnp1}, it follows that in each row and each column of $A_{n+1}$ there is also the required number of ones.

The proof of the second statement will also be carried out by induction on $n$. The base case was proved in the previous section by constructing the adjacency matrix of the digraph~$\text{dsrg}(2v+8t, k+2t, t, \lambda, t)$ in the block form given in formula~\eqref{eq:block_reprAnp1}.

Assume that $A_n^2 + sA_n = tJ_{v_n}$. Transform the expression $A_{n+1}^2 + sA_{n+1}$:
\begin{equation}
A_{n+1}^2 + sA_{n+1} = 
\left(
\begin{array}{c|c|c|c}
	A_n^2+sA_n & B_n C_n & A_n B_n + sB_n & B_n P_n \\ \hline
	B_n C_n & A_n^2+sA_n & B_n P_n & A_n B_n + sB_n \\ \hline
	P_n C_n & C_n A_n + sC_n & P_n^2 & C_n B_n + sP_n \\ \hline
	C_n A_n + sC_n & P_n C_n & C_n B_n + sP_n & P_n^2
\end{array}
\right).
\end{equation}

It is necessary to prove that all blocks of the resulting matrix are of the form $tJ$. The blocks on the main diagonal have the required form by the induction hypothesis and by Lemma~\ref{lemma:propPn}. Thus, it remains to prove the following six equalities:
\begin{align}
B_n C_n &= tJ_{v_n}; \label{eq:system1}\\
B_n P_n &= tJ_{v_n, t\cdot{4^n}}; \label{eq:system2}\\
P_n C_n &= tJ_{t\cdot{4^n}, v_n}; \label{eq:system3}\\
A_n B_n + sB_n &= tJ_{v_n, t\cdot{4^n}}; \label{eq:system4}\\
C_n A_n + sC_n &= tJ_{t\cdot{4^n}, v_n}; \label{eq:system5}\\
C_n B_n + sP_n &= tJ_{t\cdot{4^n}} \label{eq:system6}.
\end{align}

The validity of equalities~\eqref{eq:system1}--\eqref{eq:system3} follows automatically from Lemma~\ref{lemma:propsBnCnPn}. Based on the induction hypothesis, prove equalities~\eqref{eq:system4}--\eqref{eq:system6}. In the proof, Lemmas~\ref{lemma:propPn} and~\ref{lemma:Pn_newdef} will also be used.

Equality~\eqref{eq:system4}:
\begin{multline}
(A_n + s I_{v_n})B_n = \left(
\begin{array}{c|c}
I_2 \otimes (A_{n-1} + sI_{v_{n-1}}) & I_2 \otimes B_{n-1} \\ \hline
K_2 \otimes C_{n-1} & K_2 \otimes P_{n-1} + sI_{2t\cdot 4^{n-1}}
\end{array}	
\right)\cdot
\left(
\begin{array}{c}
K_2 \otimes B_{n-1} \otimes J_{1, 2} \\ \hline
I_2 \otimes P_{n-1} \otimes J_{1, 2}
\end{array}
\right)=
\\
=\left(
\begin{array}{c}
(I_2 \otimes (A_{n-1} + sI_{v_{n-1}}))\cdot (K_2 \otimes (B_{n-1} \otimes J_{1, 2})) + (I_2 \otimes B_{n-1})\cdot (I_2 \otimes (P_{n-1} \otimes J_{1, 2})) \\ \hline
(K_2 \otimes C_{n-1})\cdot (K_2 \otimes (B_{n-1} \otimes J_{1, 2})) + (K_2 \otimes P_{n-1} + sI_{2t\cdot 4^{n-1}})\cdot (I_2 \otimes (P_{n-1} \otimes J_{1, 2}))
\end{array}
\right)=
\\
=\left(
\begin{array}{c}
K_2\otimes ((A_{n-1} + sI_{v_{n-1}})\cdot (B_{n-1}\otimes J_{1, 2})) + I_2\otimes (B_{n-1}\cdot (P_{n-1}\otimes J_{1, 2})) \\ \hline
I_2\otimes (C_{n-1}\cdot (B_{n-1}\otimes J_{1,2})) + K_2\otimes (P_{n-1}\cdot (P_{n-1}\otimes J_{1, 2})) + sI_2\otimes P_{n-1}\otimes J_{1, 2} 
\end{array}
\right)=
\\
=\left(
\begin{array}{c}
K_2\otimes ((A_{n-1} + sI_{v_{n-1}})\cdot B_{n-1}) \otimes J_{1, 2} + I_2 \otimes (B_{n-1}\cdot P_{n-1}) \otimes J_{1, 2}\\ \hline
I_2\otimes(C_{n-1}B_{n-1}+sP_{n-1})\otimes J_{1, 2} + K_2\otimes (P_{n-1}\cdot P_{n-1})\otimes J_{1, 2}
\end{array}
\right)=
\\
=\left(
\begin{array}{c}
K_2\otimes tJ_{v_{n-1}, t\cdot 4^{n-1}} \otimes J_{1, 2} + I_2 \otimes tJ_{v_{n-1}, t\cdot 4^{n-1}} \otimes J_{1, 2}\\ \hline
I_2\otimes tJ_{t\cdot 4^{n-1}} \otimes J_{1, 2} + K_2\otimes tJ_{t\cdot 4^{n-1}} \otimes J_{1, 2}
\end{array}
\right)=
\left(
\begin{array}{c}
J_2\otimes tJ_{v_{n-1}, t\cdot 4^{n-1}} \otimes J_{1, 2}\\ \hline
J_2\otimes tJ_{t\cdot 4^{n-1}} \otimes J_{1, 2}
\end{array}
\right) =\\= tJ_{v_n, t\cdot{4^n}}.
\end{multline}

Equality~\eqref{eq:system5}:
\begin{multline}
C_n(A_n+sI_{v_n}) =
\left(
\begin{array}{c|c}
J_{2^n, 1} \otimes I_2 \otimes \alpha_{2^{n-1}}(C_{n-1}) & J_{2^n, 1} \otimes I_2 \otimes \alpha_{2^{n-1}}(P_{n-1})
\end{array}
\right)\times\\
\times
\left(
\begin{array}{c|c}
I_2 \otimes (A_{n-1} + sI_{v_{n-1}}) & I_2 \otimes B_{n-1} \\ \hline
K_2 \otimes C_{n-1} & K_2 \otimes P_{n-1} + sI_{2t\cdot 4^{n-1}}
\end{array}	
\right)=\\
=\left(
\begin{array}{c|c}
\text{\footnotesize $((J_{2^n, 1} \otimes I_2) \otimes \alpha_{2^{n-1}}(C_{n-1}))\cdot (I_2 \otimes (A_{n-1} + sI_{v_{n-1}}))+((J_{2^n, 1} \otimes I_2) \otimes \alpha_{2^{n-1}}(P_{n-1}))\cdot(K_2 \otimes C_{n-1})$} & ~
\end{array}
\right.
\\
\left.
\begin{array}{c|c}
~ & \text{\footnotesize $((J_{2^n, 1} \otimes I_2) \otimes \alpha_{2^{n-1}}(C_{n-1}))\cdot (I_2 \otimes B_{n-1}) + ((J_{2^n, 1} \otimes I_2) \otimes \alpha_{2^{n-1}}(P_{n-1}))\cdot(K_2 \otimes P_{n-1} + sI_{2t\cdot 4^{n-1}})$}
\end{array}
\right)=
\\
=\left(
\begin{array}{c|c}
\text{\footnotesize $((J_{2^n, 1} \otimes I_2)\cdot I_2) \otimes (\alpha_{2^{n-1}}(C_{n-1})\cdot (A_{n-1} + sI_{v_{n-1}}))+((J_{2^n, 1} \otimes I_2)\cdot K_2) \otimes (\alpha_{2^{n-1}}(P_{n-1})\cdot C_{n-1})$} & ~
\end{array}
\right.
\\
\begin{array}{c|c}
~ & \text{\footnotesize $((J_{2^n, 1} \otimes I_2)\cdot I_2) \otimes (\alpha_{2^{n-1}}(C_{n-1}) \cdot B_{n-1})+((J_{2^n, 1}\otimes I_2) \cdot K_2) \otimes (\alpha_{2^{n-1}}(P_{n-1})\cdot P_{n-1}) + $}
\end{array}
\\
\left.
\begin{array}{c}
\text{\footnotesize $ + sJ_{2^n, 1}\otimes I_2\otimes \alpha_{2^{n-1}}(P_{n-1})$}
\end{array}
\right)=
\\
=\left(
\begin{array}{c|c}
J_{2^n, 1} \otimes I_2 \otimes \alpha_{2^{n-1}}(tJ_{t\cdot 4^{n-1}, v_{n-1}}) + J_{2^n, 1} \otimes K_2 \otimes \alpha_{2^{n-1}}(tJ_{t\cdot 4^{n-1}, v_{n-1}}) & ~
\end{array}
\right.\\
\left.
\begin{array}{c|c}
~ & J_{2^n, 1} \otimes I_2 \otimes \alpha_{2^{n-1}}(tJ_{t\cdot 4^{n-1}}) + J_{2^n, 1} \otimes K_2 \otimes \alpha_{2^{n-1}}(tJ_{t\cdot 4^{n-1}})
\end{array}
\right)=\\
=\left(
\begin{array}{c|c}
J_{2^n, 1} \otimes J_2 \otimes \alpha_{2^{n-1}}(tJ_{t\cdot 4^{n-1}, v_{n-1}}) &
J_{2^n, 1} \otimes J_2 \otimes \alpha_{2^{n-1}}(tJ_{t\cdot 4^{n-1}})
\end{array}
\right)=tJ_{t\cdot{4^n}, v_n}.
\end{multline}

Equality~\eqref{eq:system6}:
\begin{multline}
C_n B_n + sP_n =
\left(
\begin{array}{c|c}
J_{2^n, 1} \otimes I_2 \otimes \alpha_{2^{n-1}}(C_{n-1}) & J_{2^n, 1} \otimes I_2 \otimes \alpha_{2^{n-1}}(P_{n-1})
\end{array}
\right)\times\\
\times
\left(
\begin{array}{c}
K_2 \otimes B_{n-1} \otimes J_{1, 2} \\ \hline
I_2 \otimes P_{n-1} \otimes J_{1, 2}
\end{array}	
\right) + sJ_{2^n, 1} \otimes K_2 \otimes \alpha_{2^{n-1}}(P_{n-1}) \otimes J_{1, 2}=\\
= ((J_{2^n, 1} \otimes I_2) \otimes \alpha_{2^{n-1}}(C_{n-1}))\cdot (K_2 \otimes (B_{n-1} \otimes J_{1, 2})) + 
\\
 + ((J_{2^n, 1} \otimes I_2) \otimes \alpha_{2^{n-1}}(P_{n-1}))\cdot (I_2 \otimes (P_{n-1} \otimes J_{1, 2})) + sJ_{2^n, 1} \otimes K_2 \otimes \alpha_{2^{n-1}}(P_{n-1}) \otimes J_{1, 2}
=\\
= ((J_{2^n, 1} \otimes I_2)\cdot K_2) \otimes (\alpha_{2^{n-1}}(C_{n-1}) \cdot (B_{n-1} \otimes J_{1, 2})) + 
\\
 + ((J_{2^n, 1}\otimes I_2)\cdot I_2)\otimes (\alpha_{2^{n-1}}(P_{n-1})\cdot (P_{n-1}\otimes J_{1,2})) + sJ_{2^n, 1} \otimes K_2 \otimes \alpha_{2^{n-1}}(P_{n-1}) \otimes J_{1, 2}
=\\
= J_{2^n, 1} \otimes K_2 \otimes \alpha_{2^{n-1}}(tJ_{t\cdot 4^{n-1}}) \otimes J_{1, 2} + J_{2^n, 1} \otimes I_2 \otimes \alpha_{2^{n-1}}(tJ_{t\cdot 4^{n-1}}) \otimes J_{1, 2}
=\\= J_{2^n, 1} \otimes J_2 \otimes \alpha_{2^{n-1}}(tJ_{t\cdot 4^{n-1}}) \otimes J_{1, 2} = tJ_{t\cdot 4^n}.
\end{multline}

\end{proof}

\section{Results of the computer program}
\label{sec3}

The authors implemented a program in the Julia language to search for families of strongly regular digraphs using the method described in the previous sections. The input to the program is the adjacency matrix $A_1$ of the first digraph in the sequence; the program then searches for a matrix $A_2$ of the required form and, if the search for $A_2$ succeeds, constructs the adjacency matrices $A_3, \ldots, A_6$ via the recurrence (see Theorem~\ref{th:main}).

To construct the matrix $A_2$, a search was performed for binary matrices $B_1$ and $C_1$ that satisfy the relations in \eqref{eq:systemA2part1} and \eqref{eq:systemA2part2} and the blockiness conditions. To simplify this search, the Artelys Knitro optimization library was used \cite{bibKnitro1} (a free trial version was employed). As the objective, the formal function $F=1$ was chosen, and all constraints imposed on $B_1$ and $C_1$ were added as constraints to this optimization problem. Thanks to the efficient algorithms of Artelys Knitro for finding a point in the feasible region, it was possible to avoid exhaustive search when looking for $B_1$ and $C_1$. The program’s source code and the adjacency matrices of the obtained digraphs are available in the repository \cite{bibGit1}.

The program results are presented below in Table~\ref{table:dsrg}. This table lists the parameters of the first two digraphs in the sequence and the parameters of the $n$-th term of the sequence. In the second column, those parameter sets are highlighted for which the existence of the corresponding digraphs was previously open (according to the table on the website \cite{bibBrouwer1} at the time of writing).

The last column of the table reports the time to find the second digraph in the sequence using the Artelys Knitro library. The program was run on a computer with an Intel Core i5-7400 (3.00 GHz) processor and 32 GB of RAM (for different parameter sets listed in the table, different Artelys Knitro settings were used).

\begin{table}[h!]
\begin{center}
\caption{Discovered families of directed strongly regular digraphs}    
\begin{tabular}{| c | c | c | c | c |} 
    \hline
    Parameters of $G_1$ & Parameters of $G_2$ & Parameters of $G_n$ & \begin{tabular}{@{}c@{}} Time to find $G_2$ \\ (sec.)\end{tabular} \\ 
    \hline & & & \\ [-1.0em]
    $(6, 3, 2, 1, 2)$ & $(28, 7, 2, 1, 2)$ & $(2^{n} ( 2^{n + 1} - 1), 2^{n+1} - 1, 2, 1, 2)$ & 3.89 \\
    $(8, 4, 3, 1, 3)$ & $(40, 10, 3, 1, 3)$ & $(2^{n} ( 3 \cdot 2^{n} - 2), 3 \cdot 2^{n} - 2, 3, 1, 3)$  & 4.03 \\
    $(10, 5, 3, 2, 3)$ & $(44, 11, 3, 2, 3)$ & $(2^{n} ( 3 \cdot 2^{n} - 1), 3 \cdot 2^{n} - 1, 3, 2, 3)$ & 9.27 \\
    $(12, 6, 4, 2, 4)$ & $(56, 14, 4, 2, 4)$ & $(2^{n} (4 \cdot 2^{n} - 2), 4 \cdot 2^{n} - 2, 4, 2, 4)$ & 16.23 \\
    $(14, 7, 4, 3, 4)$ & $(60, 15, 4, 3, 4)$ & $(2^{n} (4 \cdot 2^{n} - 1), 4 \cdot 2^{n} - 1, 4, 3, 4)$ & 2496.14 \\
    $(16, 8, 5, 3, 5)$ & \fbox{(72, 18, 5, 3, 5)} & $(2^{n} (5 \cdot 2^{n} - 2), 5 \cdot 2^{n} - 2, 5, 3, 5)$ & 52.38 \\
    $(18, 9, 5, 4, 5)$ & \fbox{(76, 19, 5, 4, 5)} & $(2^{n} (5 \cdot 2^{n} - 1), 5 \cdot 2^{n} - 1, 5, 4, 5)$ & 840.40 \\
    $(18, 9, 6, 3, 6)$ & $(84, 21, 6, 3, 6)$ & $(2^{n} (6 \cdot 2^{n} - 3), 6 \cdot 2^{n} - 3, 6, 3, 6)$ & 233.92 \\
    $(20, 10, 6, 4, 6)$ & $(88, 22, 6, 4, 6)$ & $(2^{n} (6 \cdot 2^{n} - 2), 6 \cdot 2^{n} - 2, 6, 4, 6)$ & 151.42 \\
    $(22, 11, 6, 5, 6)$ & \fbox{\:\:(92, 23, 6, 5, 6)} & $(2^{n} (6 \cdot 2^{n} - 1), 6 \cdot 2^{n} - 1, 6, 5, 6)$ & 1260.85 \\
    $(24, 12, 7, 5, 7)$ & \fbox{(104, 26, 7, 5, 7)} & $(2^{n} (7 \cdot 2^{n} - 2), 7 \cdot 2^{n} - 2, 7, 5, 7)$ & 525.10 \\ [0.3em]
    \hline
\end{tabular}
\label{table:dsrg}
\end{center}
\end{table}

It is important to note that the developed program was also run for parameter sets not shown in the table. However, in other cases the search for the matrix $A_2$ did not succeed.

It should also be noted that all directed strongly regular digraphs $G_1$ for which the method described in this work succeeded have parameters of the form $(2(t+\lambda), t+\lambda, t, \lambda, t)$. The authors do not yet know whether the method works for parameter sets of a different form, or whether such a relation between the parameters is a necessary condition for the applicability of the proposed approach.

\section{Conclusion}

The construction presented in this work makes it possible to build infinite sequences of directed strongly regular digraphs. To apply this construction, one must take the adjacency matrix $A_1$ of the first digraph in the sequence and express the adjacency matrix $A_2$ of the second digraph in terms of $A_1$ in a specific form. If this can be done, then the other digraphs in the sequence can be found using the recurrence proved in the paper. Using the described method, 11~sequences of directed strongly regular digraphs were constructed; for at least five parameter sets, the corresponding directed strongly regular digraphs were found for the first time.

The authors see promising directions for further research in determining the conditions that the digraph $G_1$ must satisfy for the method presented in the paper to work, and in finding a way to express the matrix $A_2$ in terms of the matrix $A_1$ without using a computer program.

\bibliographystyle{unsrtnat}


\end{document}